\documentclass[12pt]{amsproc}

\usepackage{fullpage}
\usepackage[utf8]{inputenc}
\usepackage[T1]{fontenc}
\usepackage{graphicx,subfigure}
\usepackage{amsmath,amsfonts,amssymb,mathtools}
\usepackage{stmaryrd}
\usepackage{mathrsfs,dsfont}
\usepackage{amsthm}
\usepackage{mathabx}
\usepackage{tabularx}
\usepackage{float}

\usepackage{hyperref}

\usepackage{enumerate}

\usepackage{color}

\newcommand{\E}{\mathbb{E}}
\newcommand{\R}{\mathbb{R}}
\newcommand{\N}{\mathbb{N}}

\newtheorem{theo}{Theorem}[section]

\newtheorem{cor}[theo]{Corollary}
\newtheorem{rem}[theo]{Remark}

\newtheorem{lemma}[theo]{Lemma}

\newtheorem{ass}{Assumption}

\usepackage{algorithm}
\usepackage{algorithmic}

\begin{document}

\title{Approximation of the invariant distribution for a class of ergodic SDEs with one-sided Lipschitz continuous drift coefficient using an explicit tamed  Euler scheme}
\author{Charles-Edouard Br\'ehier}
\address{Univ Lyon, Université Claude Bernard Lyon 1, CNRS UMR 5208, Institut Camille Jordan, 43 blvd. du 11 novembre 1918, F-69622 Villeurbanne cedex, France}
\email{brehier@math.univ-lyon1.fr}

\date{}


\begin{abstract}
We consider the long-time behavior of an explicit tamed Euler scheme applied to a class of stochastic differential equations driven by additive noise, under a one-sided Lipschitz continuity condition. The setting encompasses drift nonlinearities with polynomial growth. First, we prove that moment bounds for the numerical scheme hold, with at most polynomial dependence with respect to the time horizon. Second, we apply this result to obtain error estimates, in the weak sense, in terms of the time-step size and of the time horizon, to quantify the error to approximate averages with respect to the invariant distribution of the continuous-time process. We justify the efficiency of using the explicit tamed Euler scheme to approximate the invariant distribution, since the computational cost does not suffer from the at most polynomial growth of the moment bounds. To the best of our knowledge, this is the first result in the literature concerning the approximation of the invariant distribution for SDEs with non-globally Lipschitz coefficients using an explicit tamed scheme.
\end{abstract}

\maketitle

\section{Introduction}

The long-time behavior of deterministic and stochastic processes and of their discrete-time approximations has been an active research area, with applications in all field of science. In this manuscript, we consider Stochastic Differential Equations (SDEs) of the type
\[
dX(t)=f(X(t))dt+\sum_{k=1}^{K}\sigma_k d\beta^k(t),
\]
where $X(t)\in\R^d$, $\sigma_k\in\R^d$, $\beta^k$ are independent standard real-valued Wiener processes. The drift $f:\R^d\to\R$ is assumed to be non-globally Lipschitz continuous, instead it has polynomial growth and satisfies a one-sided Lipschitz continuous condition, see Assumptions~\ref{ass:poly} and~\ref{ass:condition}. Under these conditions, the SDE is globally well-posed and admits a unique invariant probability $\mu_\star$, such that
\[
\E[\varphi(X(T))]\underset{T\to\infty}\to \int\varphi d\mu_\star,
\]
exponentially fast, for any initial condition $X(0)$ and any real-valued Lipschitz continuous function $\varphi$. In general, no explicit expression for $f$ is known. We study the question of approximating the invariant distribution $\mu_\star$ using a numerical scheme. The main novelty of this article is to show that an explicit scheme can be used, without loss of computational efficiency, even if the nonlinearity $f$ is not globally Lipschitz continuous.

Applying an explicit scheme for SDEs with non-globally Lipschitz continuous coefficients may lead to important issues due to the lack of moment bounds, see for instance~\cite{HutzenthalerJentzenKloeden:11}. In the last two decades, many strategies have been explored: see for instance the monograph~\cite{HutzenthalerJentzen} and references therein, and the articles~\cite{HutzenthalerJentzenKloeden:12},~\cite{Sabanis},~\cite{KellyLord},~\cite{MilsteinTretyakov}, among many other contributions.

In this article, we apply the tamed Euler-Maruyama method (see Equation~\ref{eq:scheme})
\[
X_{n+1}=X_n+\frac{\Delta tf(X_n)}{1+\Delta t\|f(X_n)\|}+\sum_{k=1}^{K}\sigma_k\Delta \beta_n^k,
\]
for which for every $T$, one has moment bounds of the type
\[
\underset{0\le n\Delta t\le T}\sup~\bigl(\E[\|X_n\|^m]\bigr)^{\frac1m}<\infty.
\]
To the best of our knowledge, the question of the dependence with respect to $T$, and the possibility to have uniform in time moment bounds, has not been studied in the literature yet. Our first contribution (see Theorem~\ref{theo:moment}) is to prove that, under appropriate assumptions, one has
\[
\underset{0\le n\Delta t\le T}\sup~\bigl(\E[\|X_n\|^m]\bigr)^{\frac1m}\le C(\|x_0\|)(1+T^M),
\]
where $X_0=x_0$, for some $M\ge 1$. As a consequence, the growth with respect to $T$ is at most polynomial. We only proved upper bounds, and the question whether uniform moment bounds (with $M=0$) remains open.

Our second contribution is to apply these upper bounds for the analysis of the error between $\E[\varphi(X_N)]$ and $\int\varphi d\mu_\star$, see Theorem~\ref{theo:error}. Under appropriate assumptions, one has
\[
\big|\E[\varphi(X_N)]-\int\varphi d\mu_\star\big|\le C(\|x_0\|,\varphi)\Bigl(\exp(-\gamma N\Delta t)+\bigl(1+(N\Delta t)^M\bigr)\Delta t\Bigr),
\]
where $\gamma>0$. Contrary to existing works in the literature concerning the numerical approximation of the invariant distribution for SDEs (this is a classical problem, see for instance~\cite{MattinglyStuartHigham},~\cite{MattinglyStuartTretyakov},~\cite{Talay}), the weak error $|\E[\varphi(X_N)]-\E[\varphi(X(N\Delta t))]$ is not of size $\Delta t$ uniformly in time, and one cannot take the limit $T=N\Delta t\to \infty$ in the weak error estimate.

Nevertheless, the fact that the dependence with respect to $T=N\Delta t$ is at most polynomial is striking, and analyzing the computational cost (see Corollary~\ref{cor:cost}) shows that, to have an error less than $\varepsilon$, the cost is of size
\[
\mathcal{C}(\varepsilon)\le C\varepsilon^{-1}|\log(\varepsilon)|^{1+M}.
\]
Passing from the case with uniform in time moment bounds($M=0$) to the case with polynomial dependence ($M\ge 1$) does not substantially increase the computational cost and the method is effective. In the numerical experiments that have been performed (and are not reported here), it has not been possible to exhibit the polynomial growth with respect to time in the moment bounds or in the weak error estimate.

One of the objectives of this article is to present in a simplified framework the strategy and the arguments used in a companion paper devoted to the case of semilinear parabolic Stochastic Partial Differential Equations. The presentation is intended to be pedagogical.

This manuscript is organized as follows. The setting is presented in Section~\ref{sec:setting}, including the statement of precise assumptions (Assumptions~\ref{ass:poly} and~\ref{ass:condition}), properties of the solutions of the SDE~\eqref{eq:SDE}, and the definition of the numerical scheme (Equation~\eqref{eq:scheme}). The main results are stated and discussed in Section~\ref{sec:main}. The proof of Theorem~\ref{theo:moment} is provided in Section~\ref{sec:proof1}, whereas Section~\ref{sec:proof2} presents the proof of Theorem~\ref{theo:error}.

\section{Setting}\label{sec:setting}

Let $d\in\N$. The standard inner product and norm in the space $\R^d$ are denoted by $\langle \cdot,\cdot\rangle$ and $\|\cdot\|$ respectively.

Let us first state assumptions concerning the nonlinear drift coefficient $f$: it is assumed to be of class $\mathcal{C}^2$, with at most polynomial growth (Assumption~\ref{ass:poly}, and to satisfy a one-sided Lipschitz continuity condition (Assumption~\ref{ass:condition}).
\begin{ass}\label{ass:poly}
Let $f:\R^d\to \R^d$ be of class $\mathcal{C}^2$, with at most polynomial growth in the following sense: there exists $q\in\N$ such that
\begin{equation}
\mathcal{N}_q(f):=\underset{x\in\R^d}\sup~\frac{\|f(x)\|+\sum_{j=1}^{2}\underset{\|h_1\|,\ldots,\|h_j\|\le 1}\sup~\|D^jf(x).(h_1,\ldots,h_j)\|}{1+|x|^q}<\infty.
\end{equation}
\end{ass}
The notation $Df(x).h$ and $D^2f(x).(h_1,h_2)$ is used to denote first and second order derivatives of $f$ at point $x$, in directions $h,h_1,h_2\in\R^d$.

\begin{ass}\label{ass:condition}
The following condition is satisfied: there exists $\gamma>0$ such that for all $x_1,x_2\in\R^d$ one has
\begin{equation}\label{eq:condition}
\langle f(x_2)-f(x_1),x_2-x_1\rangle\le -\gamma|x_2-x_1|^2.
\end{equation}
\end{ass}

Let us now describe the noise. Let $\sigma_1,\ldots,\sigma_K\in \R^d$ and let $\bigl(\beta^1(t)\bigr)_{t\ge 0},\ldots\bigl(\beta^K(t)\bigr)_{t\ge 0}$ be independent standard real-valued Wiener processes, defined on a probability space $(\Omega,\mathcal{F},\mathbb{P})$ satisfying the usual conditions. In the sequel, the notation
\[
\sigma dB(t)=\sum_{k=1}^{K}\sigma_k d\beta^k(t)
\]
is used.

In this work, we consider the following SDE with values in $\R^d$:
\begin{equation}\label{eq:SDE}
dX(t)=f(X(t))dt+\sigma dB(t).
\end{equation}
Owing to the locally Lipschitz and the one-sided Lipschitz properties of the function $f$, it is straigthforward to check that for any $x_0\in \R^d$, there exists a unique global solution $\bigl(X_{x_0}(t)\bigr)_{t\ge 0}$ to~\eqref{eq:SDE} with $X_{x_0}(0)=x_0$. Moreover, for every $m\in[1,\infty)$, there exists a polynomial function $\mathcal{P}_m:\R\to\R$ such that
\begin{equation}\label{eq:momentX}
\underset{t\ge 0}\sup~\E[\|X_{x_0}(t)\|^m]\le \mathcal{P}_m(\|x_0\|).
\end{equation}
Finally, there exists a unique invariant distribution $\mu_\star$, such that any Lipschitz continuous function $\varphi:\R^d\to\R$, one has for all $T\in(0,\infty)$ and $x_0\in\R^d$
\begin{equation}\label{eq:ergoX}
\big|\E[\varphi(X_{x_0}(T))]-\int\varphi d\mu_\star\big|\le e^{-\gamma T}{\rm Lip}(\varphi)(1+\|x_0\|).
\end{equation}
Moment bounds~\eqref{eq:momentX} are obtained by an application of the It\^o formula, and the use of the one-sided Lipschitz continuity condition~\eqref{eq:condition}. Applying the Krylov-Bogoliubov criterion, the moment bounds~\eqref{eq:momentX} yield existence of invariant distributions, and uniqueness follows from a straightforward coupling argument: if $x_0^1,x_0^2\in\R^d$ are two initial conditions, then almost surely for all $t\ge 0$ one has
\[
\|X_{x_0^2}(t)-X_{x_0^1}(t)\|\le e^{-\gamma t}\|x_0^2-x_0^1\|,
\]
using the one-sided Lipschitz condition~\eqref{eq:condition}, where the processes $X_{x_0^1}$ and $X_{x_0^2}$ are driven by the same noise process.

In the sequel, to simplify notation the dependence with respect to $x_0$ is omitted: $X(t)=X_{x_0}(t)$ for all $t\ge 0$.

The objective is to define estimators of $\int\varphi d\mu_\star$, using an explicit integrator for the discretization of the SDE~\eqref{eq:SDE}. Note that in general, no explicit expression of the invariant distribution $\mu_\star$ is known, and if dimension $d$ is large quadrature rules are not efficient and Monte-Carlo methods are usually employed.

Let $\Delta t$ denote the time-step size, and without loss of generality assume that $\Delta t\in(0,\Delta t_0]$ for some arbitrarily fixed $\Delta t_0>0$. Let $t_n=n\Delta t$, and define the Wiener increments as $\Delta \beta_n^k=\beta^k(t_{n+1})-\beta^k(t_n)$, and $\sigma\Delta B_n=\sigma_k\Delta t\beta_n^k$, for all $n\ge 0$.
%

In this work, we consider the explicit tamed Euler-Maruyama method
\begin{equation}\label{eq:scheme}
X_{n+1}=X_n+\frac{\Delta t}{1+\alpha\Delta t\|f(X_n)\|}f(X_n)+\sigma\Delta B_n,
\end{equation}
where $\alpha\in(0,\infty)$ is a given parameter and $X_0=x_0$. The parameter $\alpha$ does not play an important role in the sequel, and the notation $M_n=\alpha\|f(X_n)\|$ is used. Introduce a continuous-time auxiliary process $\bigl(\tilde{X}(t)\bigr)_{t\ge 0}$ as follows: for all $t\ge 0$, set
\begin{equation}\label{eq:tildeX}
\tilde{X}(t)=x_0+\int_{0}^{t}\frac{1}{1+\Delta tM_{\ell(s)}}f(X_{\ell(s)})ds+\int_{0}^{t}\sigma dB(s)
\end{equation}
where $\ell(t)=n$ if and only if $t_n\le t<t_{n+1}$, and $\bigl(X_n\bigr)_{n\ge 0}$ is defined by the tamed explicit Euler-Maruyama scheme~\eqref{eq:scheme}. Observe that one has $\tilde{X}(t_n)=X_n$ for all $n\ge 0$.

The process $\bigl(X_n\bigr)_{n\ge 0}$ and $\bigl(\tilde{X}(t)\bigr)_{t\ge 0}$ depend on the time-step size $\Delta t$ and on the initial condition $x_0$, however the dependence is omitted to simplify notation.

\begin{rem}
The explicit tamed Euler-Maruyama scheme~\eqref{eq:scheme} can be interpreted as coming from the application of the standard explicit Euler-Maruyama scheme for the SDE
\[
dX^{\Delta t}(t)=f_{\Delta t}(X^{\Delta t}(t))dt+\sigma dB(t)
\]
with modified drift coefficient
\[
f_{\Delta t}(x)=\frac{f(x)}{1+\alpha\|f(x)\|}.
\]
However, the function $f^{\Delta t}$ does not satisfy a one-sided Lipschitz continuity condition as in Assumption~\ref{ass:condition}, with $\gamma>0$. Indeed, let $d=1$ and assume that $f$ is a polynomial function, for instance $f(x)=-x-x^3$, then one checks that
\[
f_{\Delta t}'(x)\underset{x\to\pm\infty}\to 0.
\]
As a consequence, the long-time behavior (uniform in time moment bounds, existence and uniqueness of invariant distriutions) of the process $X^{\Delta t}$ may be non trivial. It may even be the case that $\underset{x\in\R}\sup~f_{\Delta t}'(x)>0$, in which case uniform in time moment bounds are not expected to hold. In turn no information on the behavior of the scheme~\eqref{eq:scheme} is obtained by using the interpretation of the tamed scheme.
\end{rem}

\section{Main results}\label{sec:main}

The objective of this section is to state the two main results of this article: Theorem~\ref{theo:moment} concerning moment bounds and Theorem~\ref{theo:error} concering weak error estimates. The most striking feature is the polynomial dependence with respect to the time $T$ (and with respect to the norm of the initial condition). The consequences in terms of computational cost for the approximation of the invariant distributions are also discussed, see Corollary~\ref{cor:cost}.

\begin{theo}\label{theo:moment}
Let Assumptions~\ref{ass:poly} and~\ref{ass:condition} be satisfied.

For every $m\in[1,\infty)$, there exists a polynomial function $\mathcal{P}_m:\R^2\to \R$, such that for all $T\in(0,\infty)$ and $x_0\in\R^d$, one has
\begin{equation}
\underset{\Delta t\in(0,\Delta t_0]}\sup~\underset{0\le n\Delta t\le T}\sup~\E[\|X_m\|^p]\le \mathcal{P}_m(T,\|x_0\|).
\end{equation}
\end{theo}

\begin{theo}\label{theo:error}
Let Assumptions~\ref{ass:poly} and~\ref{ass:condition} be satisfied.

For every $r\in\N$, there exists polynomial functions $\mathcal{P}_r:\R^2\to\R$ and $\mathcal{Q}_r:\R\to\R$, such that the following holds: if $\varphi:\R^d\to\R$ is a function of class $\mathcal{C}^2$ such that
\[
\mathcal{N}_r(\varphi)=\underset{x\in\R^d}\sup~\frac{|\varphi(x)|+\sum_{j=1}^{2}\underset{\|h_1\|,\ldots,\|h_j\|\le 1}\sup~|D^jf(x).(h_1,\ldots,h_j)|}{1+|x|^r}<\infty,
\]
then for all $N\in\N$, $x_0\in\R^d$ and $\Delta t\in(0,\Delta t_0]$, one has
\begin{equation}
\big|\E[\varphi(X_N)]-\int\varphi d\mu_\star\big|\le \mathcal{N}_r(\varphi)\Bigl(\exp(-\gamma N\Delta t)Q_r(\|x_0\|)+\Delta t\mathcal{P}_r(N\Delta t,\|x_0\|)\Bigr),
\end{equation}
where $\bigl(X_n\bigr)_{n\ge 0}$ is defined by the tamed Euler-Maruyama scheme~\eqref{eq:scheme}.
\end{theo}
In Theorem~\ref{theo:error}, the function $\varphi$ is allowed to have at most polynomial growth (as well as its derivatives).

To simplify the presentation, the degrees of the polynomial functions $\mathcal{P}_m$ in Theorem~\ref{theo:moment} and $\mathcal{P}_r$ and $\mathcal{Q}_r$ in Theorem~\ref{theo:error} are not indicated, however they could be identified by a close inspection of the proofs. This is left to the interested readers. Since we want to insist on the polynomial dependence with respect to time $T$ or $N\Delta t$, the exact value of the degrees does not matter.

Let us now draw consequences of Theorem~\ref{theo:error} in terms of computational cost. In practice, the expectation $\E[\varphi(X_n)]$ needs to be approximated using Monte-Carlo averages. The total computational cost may be reduced for instance using the Multilevel Monte-Carlo method. Since those aspects are not specific to the situation studied in this article, we only consider the cost per realization. The cost to sample a realization of $X_N$ is proportional to $N$. Then one has the following result.

\begin{cor}\label{cor:cost}
Let $x_0\in\R^d$ and $\varphi:\R^d\to\R$ of class $\mathcal{C}^2$ such that $\mathcal{N}_r(\varphi)<\infty$.

There exists a constant $C=C(x_0,\varphi)\in(0,\infty)$, such that for all $\varepsilon\in(0,1)$, the error satisfies
\[
\big|\E[\varphi(X_N)]-\int\varphi d\mu_\star\big|\le \varepsilon
\]
with a computational cost satisfying
\[
\mathcal{C}(\varepsilon)\le C\varepsilon^{-1}|\log(\varepsilon)|^{1+R}
\]
for some $R\in\N$.
\end{cor}

Since $|\log(\varepsilon)|^R=\varepsilon^{\frac1R}$ for all $R\in\N$, one has $\mathcal{C}(\varepsilon)\le C_\alpha\varepsilon^{-\frac{1}{\alpha}}$, for all $\alpha\in(0,1)$.

If in Theorems~\ref{theo:moment} and~\ref{theo:error} the upper bounds were uniform in time, one could choose $R=0$ in Corollary~\ref{cor:cost}. The polynomial growth with respect to $T$ only introduces a polynomial factor in $|\log(\varepsilon)|$, which does not significantly increase the computational cost. Note that if the dependence with respect to $T$ in Theorems~\ref{theo:moment} and~\ref{theo:error} had been exponential, an additional factor which would have been a power of $\varepsilon^{-1}$ would have appeared in Corollary~\ref{cor:cost}, and in turn the computational cost would have significantly increased.

\begin{proof}[Proof of Corollary~\ref{cor:cost}]
The error estimate of Theorem~\ref{theo:error} is rewritten as follows: there exists $R\in\N$ such that
\[
\big|\E[\varphi(X_N)]-\int\varphi d\mu_\star\big|\le C\bigl(\exp(-\gamma N\Delta t)+(1+(N\Delta t)^R)\Delta t\bigr),
\]
where $C$ is a constant (depending on $x_0$ and $\varphi$). Note that $R$ may depend on $\varphi$ (more precisely on the value of $r$ such that $\mathcal{N}_r(\varphi)<\infty$.

The parameters $N$ and $\Delta t$ are chosen such that
\[
N\Delta t=C|\log(\varepsilon)|
\]
and
\[
(N\Delta t)^{R}\Delta t=C\varepsilon.
\]

This leads to
\[
\Delta t=C\varepsilon |\log(\varepsilon)|^{-R},
\]
and finally
\[
N=C|\log(\varepsilon)|\Delta t^{-1}=C\varepsilon^{-1}|\log(\varepsilon)|^{1+R}.
\]
This concludes the proof of Corollary~\ref{cor:cost}.
\end{proof}

\begin{rem}
In~\cite{MilsteinTretyakov}, the authors propose to use the so-called rejecting exploding trajectories technique to approximate ergordic averages $\int \varphi d\mu_\star$, for SDEs with non-globally Lipschitz coefficients. This technique requires to introduce an auxiliary truncation parameter. However, even if in practice it is effective, this technique does not lead to a clean analysis of the cost as in Corollary~\ref{cor:cost}.
\end{rem}

\begin{rem}\label{rem:multi}
The results can be generalized, under appropriate assumptions, to SDEs driven by multiplicative noise
\[
dX(t)=f(X(t))dt+\sigma(X(t))dB(t),
\]
where $\sigma(x)dB(t)=\sum_{k=1}^{K}\sigma_k(x)d\beta_k(t)$. Slight modifications of the proof of Theorem~\ref{theo:moment} are sufficient if the functions $\sigma_k$, $k\in\{1,\ldots,K\}$, are assumed to be bounded. These functions are also to be Lipschitz continuous. The one-sided Lipschitz condition~\eqref{eq:condition} from Assumption~\ref{ass:condition} needs to be modified to ensure the uniqueness of the invariant distribution by the coupling argument providing~\eqref{eq:ergoX}. The condition~\eqref{eq:condition} also needs to be modified in order to be able to prove Theorem~\ref{theo:error}, in particular to prove exponential decrease in time of derivatives of the solution of the associated Kolmogorov equation.

These modifications and the statements of appropriate conditions would only make the presentation more complex, so for pedagogical reasons the details are only provided for the additive noise case.
\end{rem}

\section{Proof of Theorem~\ref{theo:moment}}\label{sec:proof1}

Recall that $M_n=\alpha\|f(X_n)\|$, and that the auxiliary process $\bigl(\tilde{X}(t)\bigr)_{t\ge 0}$ is defined by~\eqref{eq:tildeX}.

Introduce an auxiliary parameter $R=\Delta t^{-\kappa}$, where $\kappa\in(0,\frac{1}{2q})$ and $q$ is given by Assumption~\ref{ass:poly}.

For every $n\ge 0$, let $\Omega_{R,t_{n}}=\{\underset{0\le \ell\le n}\sup~\|X_\ell\|\le R\}$, and to simplify notation let $\chi_n=\mathds{1}_{\Omega_{R,t_n}}$ denote the indicator function of the set $\Omega_{R,t_n}$. Let also $\chi_{-1}=1$.

To prove Theorem~\ref{theo:moment}, it suffices to prove the following two auxiliary results.
\begin{lemma}\label{lemma1}

For every $m\in[1,\infty)$, there exists a polynomial function $\mathcal{P}_m^1:\R^2\to \R$, such that for all $T\in(0,\infty)$ and $x_0\in\R^d$, one has
\begin{equation}\label{eq:lemma1}
\underset{\Delta t\in(0,\Delta t_0]}\sup~\underset{0\le n\Delta t\le T}\sup~\E[\chi_{n-1}\|X_n\|^m]\le \mathcal{P}_m^1(T,\|x_0\|).
\end{equation}
\end{lemma}

\begin{lemma}\label{lemma2}

For every $m\in[1,\infty)$, there exists a polynomial function $\mathcal{P}_m^2:\R^2\to \R$, such that for all $T\in(0,\infty)$ and $x_0\in\R^d$, one has
\begin{equation}\label{eq:lemma2}
\underset{\Delta t\in(0,\Delta t_0]}\sup~\underset{0\le n\Delta t\le T}\sup~\E[(1-\chi_{n})\|X_n\|^m]\le \mathcal{P}_m^2(T,\|x_0\|).
\end{equation}
\end{lemma}

Theorem~\ref{theo:moment} is then a straightforward consequence of Lemma~\ref{lemma1} and Lemma~\ref{lemma2}.
\begin{proof}[Proof of Theorem~\ref{theo:moment}]
Since $\Omega_{R,t_n}\subset\Omega_{R,t_{n-1}}$, one has $\chi_n\le \chi_{n-1}$. Writing
\[
\E[\|X_n\|^m]=\E[\chi_n\|X_n\|^m]+\E[(1-\chi_n)\|X_n\|^m]\le \E[\chi_{n-1}\|X_n\|^m]+\E[(1-\chi_n)\|X_n\|^m],
\]
combining the auxiliary moment bounds~\eqref{eq:lemma1} and~\eqref{eq:lemma2} then concludes the proof.
\end{proof}

In the proofs, values of constants $C,C_m\in(0,\infty)$ and polynomial functions $\mathcal{P}_m$ may change from line to line.

\begin{proof}[Proof of Lemma~\ref{lemma1}]
Introduce two auxiliary processes $\bigl(Y(t)\bigr)_{t\ge 0}$ and $\bigl(Z(t)\bigr)_{t\ge 0}$ as follows: for all $t\ge 0$,
\begin{align*}
Z(t)&:=\tilde{X}(t)-\int_0^t f(\tilde{X}(s))ds\\
Y(t)&:=\tilde{X}(t)-Z(t).
\end{align*}
One obtains the following equality: for all $t\ge 0$
\[
Y(t)=\int_0^t f(\tilde{X}(s))ds=\int_0^tf(Y(s)+Z(s))ds,
\]
thus $Y$ solves the differential equation
\[
\frac{dY(t)}{dt}=f(Y(t)+Z(t)).
\]
Using the one-sided Lipschitz condition~\eqref{eq:condition} satisfied by $f$ (Assumption~\ref{ass:condition}), one obtains
\begin{align*}
\frac12\frac{d\|Y(t)\|^2}{dt}&=\langle Y(t),f(Y(t)+Z(t)\rangle\\
&\le \langle Y(t),f(Y(t)+Z(t))-f(Z(t))\rangle+\|Y(t)\|\|f(Z(t))\|\\
&\le -\gamma\|Y(t)\|^2+\|Y(t)\|f(Z(t))\|.
\end{align*}
Since $\gamma>0$, using Young's inequality and Gronwall's lemma, one obtains
\[
\|Y(t)\|^2\le C\int_0^t \|f(Z(s))\|^2ds.
\]
Multiplying by $\chi_{n-1}$ and using Minkowski's inequality, for every $m\in[1,\infty)$, one obtains
\begin{equation}\label{eq:momentY}
\bigl(\E[\chi_{n-1}\|Y(t_n)\|^{2m}]\bigr)^{\frac1m}\le C\int_0^t \bigl(\E[\chi_{n-1}\|f(Z(s))\|^{2m}]\bigr)^{\frac1m}ds.
\end{equation}
Since $f$ has at most polymomial growth (Assumption~\ref{ass:poly}), oo obtain~\eqref{eq:lemma1}, it is thus sufficient to prove that for every $m\in[1,\infty)$, one has an estimate of the type
\begin{equation}\label{eq:momentZ}
\underset{0\le t\le t_n\le T}\sup~\E[\chi_{n-1}\|Z(t)\|^{m}]\le \mathcal{P}_m(T,\|x_0\|).
\end{equation}

By definition of the auxiliary process $Z$, one has for all $t\ge 0$
\begin{align*}
Z(t)&=\tilde{X}(t)-\int_{0}^{t}f(\tilde{X}(s))ds\\
&=X_0+\int_0^t\sigma dB(s)-\int_0^t\frac{\Delta tM_{\ell(s}}{1+\Delta tM_{\ell(s)}}f(X_{\ell(s)})ds+\int_0^t\bigl[f(X_{\ell(s)})-f(\tilde{X}(s))\bigr]ds\\
&=:Z_0(t)+Z_1(t)+Z_2(t).
\end{align*}

First, for all $0\le t\le t_n\le T$, one has
\[
\bigl(\E[\chi_{n-1}\|Z_0(t)\|^m]\bigr)^{\frac1m}\le \bigl(\E[Z_0(t)\|^m\bigr)^{\frac1m} \le \|x_0\|+C_mT^{\frac{1}{2}},
\]
for some $C_m\in(0,\infty)$.

Second, recall that $M_\ell=\alpha \|f(X_\ell)\|$. Since $f$ has at most polynomial growth (Assumption~\ref{ass:poly}), one obtains
\begin{align*}
\bigl(\E[\chi_{n-1}\|Z_1(t)\|^m]\bigr)^{\frac1m}&\le \alpha\Delta t\int_0^t\bigl(\E[\chi_{n-1}\|f(X_{\ell(s)})\|^{2m}]\bigr)^{\frac1m}ds\\
&\le C \Delta t(1+R^{2q})\le C (\Delta t_0+\Delta t_0^{1-2q\kappa}),
\end{align*}
where we recall that $R=\Delta t^{-\kappa}$ with $2q\kappa<1$.

It remains to deal with the term $\E[\chi_{n-1}\|Z_2(t)\|^m]$. Using the fact that the derivative of $f$ has at most polynomial growth (Assumption~\ref{ass:poly}) and Minkowski and Cauchy-Schwarz inequalities, one obtains for all $t\ge 0$
\begin{align*}
\bigl(&\E[\chi_{n-1}\|Z_2(t)\|^m]\bigr)^{\frac1m}\le \int_0^t\bigl(\chi_{n-1}\|f(X_{\ell(s)}-f(\tilde{X}(s)\|^m\bigr)^{\frac1m}ds\\
&\le C\int_{0}^t \bigl(\E[\chi_{n-1}\|X_{\ell(s)}-\tilde{X}(s)\|^{2m}]\bigr)^{\frac{1}{2m}}\bigl(1+\E[\chi_{n-1}\|X_{\ell(s)}\|^{2mq}]+\E[\chi_{n-1}\|\tilde{X}(s)\|^{2mq}]\bigr)^{\frac{1}{2m}} ds.
\end{align*}
The first factor in the integrand above is treated as follows: for $s\le t<t_n$, one obtains
\begin{align*}
\chi_{n-1}\|X_{\ell(s)}-\tilde{X}(s)\|&\le \chi_{n-1}\frac{|s-t_{\ell(s)}|}{1+\Delta tM_{\ell(s)}}\|f(X_{\ell(s)}\|+C\|\sigma\bigl(B(s)-B(t_{\ell(s)})\bigr)\|\\
&\le C\Delta t(1+R^q)+C\|B(s)-B_{t_{\ell(s)}}\|.
\end{align*}
As a consequence, one obtains
\[
\bigl(\E[\chi_{n-1}\|X_{\ell(s)}-\tilde{X}(s)\|^{2p}]\bigr)^{\frac{1}{2p}}\le C_p\bigl(\Delta t(1+R^q)+\Delta t^{\frac12}\bigr)\le C\Delta t^{\frac12}(\Delta t_0^{\frac{1}{2}-\kappa q}+1),
\]
for all $\Delta t\in(0,\Delta t_0]$, using the definition $R=\Delta t^{-\kappa}$ with $2q\kappa<1$.

To treat the second factor of the integrand above, it suffices to write
\[
\E[\chi_{n-1}\|X_{\ell(s)}\|^{2mq}]+\E[\chi_{n-1}\|\tilde{X}(s)\|^{2mq}\le C\E[\chi_{n-1}\|X_{\ell(s)}\|^{2mq}]+C\E[\chi_{n-1}\|\tilde{X}(s)-X_{\ell(s)}\|^{2mq},
\]
and to use the estimate on the first factor above and the inequality $\E[\chi_{n-1}\|X_{\ell(s)}\|^{2mq}]\le R^{2mq}$.

Finally, using again the inequality $\Delta t^{\frac12}R^q\le \Delta t_0^{\frac12-q\kappa}$ for all $\Delta t\in(0,\Delta t_0]$, one obtains
\[
\bigl(\E[\chi_{n-1}\|Z_2(t)\|^m]\bigr)^{\frac1m}\le C(\Delta t_0)T.
\]

Gathering the estimates then yields~\eqref{eq:momentZ}. Inserting~\eqref{eq:momentZ} in the inequality~\eqref{eq:momentY} then yields
\[
\E[\chi_{n-1}\|Y(t_n)\|^{m}]\le \mathcal{P}_m(T,\|x_0\|),
\]
if $t_n\le T$ and $\Delta t\in(0,\Delta t_0]$.

Since $X_n=\tilde{X}(t_n)=Y(t_n)+Z(t_n)$, this concludes the proof of Lemma~\ref{lemma1}.
\end{proof}

\begin{proof}[Proof of Lemma~\ref{lemma2}]

Recall that $\chi_n=\mathds{1}_{\Omega_{R,t_{n}}}$, with $\Omega_{R,t_n}=\{\underset{0\le \ell\le n}\sup~\|X_\ell\|\le R\}$ and $\chi_{-1}=1$. As a consequence, one has
\begin{align*}
1-\chi_n&=\mathds{1}_{\Omega_{R,t_n}^c}=\mathds{1}_{\Omega_{R,t_{n-1}}^c}+\mathds{1}_{\Omega_{R,t_{n-1}}}\mathds{1}_{\|X_n\|>R}\\
&=1-\chi_{n-1}+\chi_{n-1}\mathds{1}_{\|X_n\|>R}.
\end{align*}
One thus obtains the equality
\[
1-\chi_n=\sum_{\ell=0}^{n}\chi_{\ell-1}\mathds{1}_{\|X_\ell\|>R}.
\]
Let $m\in\N$. Using Minkowksi, Cauchy-Schwarz and Markov inequalities, one obtains
\begin{align*}
\bigl(\E[(1-\chi_n)\|X_n\|^{m}]\bigr)^{\frac1m}&\le \sum_{\ell=0}^{n}\bigl(\E[\chi_{\ell-1}\mathds{1}_{\|X_\ell\|>R}\|X_n\|^m]\bigr)^{\frac1m}\\
&\le \sum_{\ell=0}^{n}\bigl(\E[\|X_n\|^{2m}]\bigr)^{\frac{1}{2m}}\bigl(\E[\chi_{\ell-1}\frac{\|X_\ell\|^\theta}{R^\theta}]\bigr)^{\frac{1}{2m}},
\end{align*}
where $\theta\in\N$ is chosen below.

On the one hand, by construction of the tamed Euler scheme, one has
\[
\|\tilde{X}(t)\|\le \|x_0\|+\frac{T}{\Delta t}+\|\int_0^t\sigma(\tilde{X}(s))dB(s)\|,
\]
thus 
\[
\bigl(\E[\|X_n\|^{2m}]\bigr)^{\frac{1}{2m}}\le C(\|x_0\|+T^{\frac12}+\frac{T}{\Delta t}).
\]
On the other hand, applying Lemma~\ref{lemma1} yields
\[
\E[\chi_{\ell-1}\|X_\ell\|^\theta]\le \mathcal{P}_\theta^1(T,\|x_0\|).
\]
Gathering the estimates yields
\[
\bigl(\E[(1-\chi_n)\|X_n\|^{m}]\bigr)^{\frac1m}\le C\mathcal{P}_\theta^1(T,\|x_0\|)^{\frac{1}{2m}}(1+\frac{T}{\Delta t})(\|x_0\|+T^{\frac12}+\frac{T}{\Delta t})R^{-\frac{\theta}{2m}}.
\]
Since $R=\Delta t^{-\kappa}$, it suffices to choose $\theta$ such that $\frac{\theta\kappa}{2m}>2$ in order to obtain~\eqref{eq:lemma2}.

This concludes the proof of Lemma~\ref{lemma2}.

\end{proof}

\section{Proof of Theorem~\ref{theo:error}}\label{sec:proof2}

The objective of this section is to prove Theorem~\ref{theo:error}. In Section~\ref{sec:Kolmo}, some auxiliary results concerning the solution of the associated Kolmogorov equation. Even if this type of results is standard, a proof is provided for completeness. Then Theorem~\ref{theo:error} follows from the weak error analysis of Section~\ref{sec:error}.

Like in Section~\ref{sec:proof1}, the values of constants $C\in(0,\infty)$ and of polynomial functions $\mathcal{P}_r$ or $\mathcal{P}$ may change from line to line.

\subsection{Auxiliary result: Kolmogorov equation}\label{sec:Kolmo}

Let $u(t,x)=\E_x[\varphi(X_t)]$, for all $t\ge 0$ and $x\in\R^d$. Then $u$ is the solution of the Kolmogorov equation
\[
\partial_tu(t,x)=\mathcal{L}u(t,x)=Du(t,x).f(x)+\frac{1}{2}\sum_{k=1}^{K}D^2u(t,x).(\sigma_k,\sigma_k)
\]
with initial condition $u(0,\cdot)=\varphi$. Set $\overline{u}(t,\cdot)-\int \varphi d\mu_\star$.

The objective of this section is to prove the following lemma.
\begin{lemma}\label{lem:Kolmogorov}
Let Assumptions~\ref{ass:poly} and~\ref{ass:condition} be satisfied.

For any $r\in\N$, there exists $R\in\N$ such that the following holds: for all $\varphi:\R^d\to\R$ of class $\mathcal{C}^2$ which satisfies $\mathcal{N}_r(\varphi)<\infty)$, for all $\gamma'\in[0,\gamma)$, one has
\[
\underset{t\ge 0}\sup~\mathcal{N}_{R}(\overline{u}(t,\cdot))e^{\gamma' t}<\infty.
\]
\end{lemma}
In other words, the function $\overline{u}$ and its first and second order spatial derivatives $D\overline{u}(t,\cdot)$ and $D\overline{u}(t,\cdot)$ have at most polynomial growth, and $\mathcal{N}_{R}(\overline{u}(t,\cdot))$ decreases exponentially fast when $t\to\infty$. This type of result and the strategy of the proof are standard. A specific feature of the approach considered in this article is that a weaker result is employed: since we do not expect to prove uniform in time weak error estimates, and to have polynomial in time dependence instead, it is sufficient to use uniform in time upper bounds for the derivatives $D\overline{u}(t,\cdot)$ and $D\overline{u}(t,\cdot)$. Thus in the analysis of the weak error due to temporal discretization, one will use estimates with $\gamma'=0$. Keeping $\gamma'>0$ would only allow to reduce the degree of the polynomial dependence with respect to $T$, but this does not qualitatively change the analysis of the cost. The fact that $\gamma'>0$ has a role only to estimate the error $\overline{u}(t,x)=u(t,x)-\int\varphi d\mu_\star$, and as will be clear below for this contribution one can take $\gamma'=\gamma$ (see also~\eqref{eq:ergoX}).

\begin{proof}
Recall that, for every $x\in\R^d$, the solution of~\eqref{eq:SDE} with initial condition $X(0)=x$ is denoted by $\bigl(X_x(t)\bigr)_{t\ge 0}$.

First, let $x_1,x_2\in\R^d$, then the one-sided Lipschitz condition~\eqref{eq:condition} implies
\[
\frac{1}{2}\frac{d\|X^{x_2}(t)-X^{x_1}(t)\|^2}{dt}=\langle f(X^{x_2}(t))-f(X^{x_1}(t)),X^{x_2}(t)-X^{x_1}(t)\rangle\le -\gamma\|X^{x_2}(t)-X^{x_1}(t)\|^2,
\]
thus $\|X^{x_2}(t)-X^{x_1}(t)\|\le e^{-\gamma t}\|x_2-x_1\|$ for all $t\ge 0$. One obtains for all $t\ge 0$
\begin{align*}
\big|\E[\varphi(X^{x_2}(t))]-\E[\varphi(X^{x_1}(t))]\big|&\le C\mathcal{N}_r(\varphi)\E\bigl[\bigl(1+\|X^{x_1}(t)\|^r+\|X^{x_2}(t)\|\bigr)\|X^{x_2}(t)-X^{x_1}(t)\|\bigr]\\
&\le C\mathcal{N}_r(\varphi)\bigl(1+\mathcal{P}_r(\|x_1\|)+\mathcal{P}_r(\|x_2\|)\bigr)e^{-\gamma t}\|x_2-x_1\|.
\end{align*}
Choosing $x_2=x$ and integrating with respect to $d\mu_\star(x_2)$ then yields
\[
|\overline{u}(t,x)|\le C\mathcal{N}_r(\varphi)e^{-\gamma t}(1+\|x\|\mathcal{P}_r(\|x\|)),
\]
using the moment bound~\eqref{eq:momentX} for the exact solution.

Next, the first-order derivative of $\overline{u}(t,\cdot)$ satisfies
\[
D\overline{u}(t,x).h=\E[D\varphi(X^x(t)).\eta^h(t)],
\]
where one has
\[
\frac{d\eta^h(t)}{dt}=Df(X^x(t)).\eta^h(t)~,\quad \eta^h(0)h.
\]
Using the one-sided Lipschitz condition~\eqref{eq:condition} implies
\[
\frac{1}{2}\frac{d\|\eta^h(t)\|^2}{dt}=\langle Df(X^x(t)),\eta^h(t)\rangle\le -\gamma\|\eta^h(t)\|^2,
\]
thus Gronwall's lemma yields the inequality $\|\eta^h(t)\|\le e^{-\gamma t}\|h\|$ for all $t\ge 0$. One obtains
\[
|D\overline{u}(t,x).h|\le e^{-\gamma t}\|h\|\mathcal{N}_r(\varphi)\E[1+\|X^x(t)\|^r]\le e^{-\gamma t}\|h\|\mathcal{N}_r(\varphi)\bigl(1+\mathcal{P}_r(\|x\|)\bigr),
\]
using the moment bound~\eqref{eq:momentX} for the exact solution.

Finally, the second-order derivative of $\overline{u}(t,\cdot)$ satisfies
\[
D^2\overline{u}(t,x).(h_1,h_2)=\E[D\varphi(X^x(t)).\zeta^{h_1,h_2}(t)]+\E[D^2\varphi(X^x(t)).(\eta^{h_1}(t),\eta^{h_2}(t))],
\]
where one has
\[
\frac{d\zeta^{h_1,h_2}(t)}{dt}=Df(X^x(t)).\zeta^{h_1,h_2}(t)+D^2f(X^x(t)).(\eta^{h_1}(t),\eta^{h_2}(t))~,\quad \zeta^{h_1,h_2}(0)=0.
\]
Using the one-sided Lipschitz condition~\eqref{eq:condition}, the at most polynomial growth of $f$ and its derivatives (Assumption~\ref{ass:poly}) and Young's inequality, one has
\begin{align*}
\frac{1}{2}\frac{d\|\zeta^{h_1,h_2}(t)\|^2}{dt}&=\langle Df(X^x(t)).\zeta^{h_1,h_2}(t),\zeta^{h_1,h_2}(t)\rangle+\langle D^2f(X^x(t)).(\eta^{h_1}(t),\eta^{h_2}(t)),\zeta^{h_1,h_2}(t)\rangle\\
&\le -(\gamma+\epsilon)\|\zeta^{h_1,h_2}(t)\|^2+\frac{1}{4\epsilon}\mathcal{N}_{q}(f)^2(1+\|X^x(t)\|^{2q})e^{-4\gamma t}\|h_1\|^2\|h_2\|^2.
\end{align*}
Using Gronwall's lemma and the moment bound~\eqref{eq:momentX} for the exact solution, one obtains
\[
\E[\|\zeta^{h_1,h_2}(t)\|^2]\le Ce^{-2\gamma't}\|h_1\|^2\|h_2\|^2\bigl(1+\mathcal{P}_{2q}(\|x\|)\bigr).
\]
Finally, this gives the inequalities
\begin{align*}
\big|D^2\overline{u}(t,x).(h_1,h_2)\big|&\le \mathcal{N}_r(\varphi)\bigl(\E[(1+\|X^x(t)\|^{r})^2]\bigr)^{\frac12}\bigl(\E\|\zeta^{h_1,h_2}(t)\|^2\bigr)^{\frac12}\\
&~+\mathcal{N}_r(\varphi)\E[(1+\|X^x(t)\|)^r)]e^{-2\gamma t}\|h_1\|\|h_2\|\\
&\le Ce^{-\gamma' t}\|h_1\|\|h_2\|\mathcal{N}_r(\varphi)(1+\mathcal{P}_{2r}(\|x\|)).
\end{align*}
Gathering the estimates then concludes the proof.

\end{proof}

\subsection{Weak error analysis}\label{sec:error}

By a linearity argument, without loss of generality one assumes that $\mathcal{N}_r(\varphi)\le 1$.

\begin{proof}[Proof of Theorem~\ref{theo:error}]

The weak error can be written as
\begin{align*}
\E[\varphi(X_N)]-\int\varphi d\mu_\star&=\E[\overline{u}(0,X_N)]\\
&=\E[\overline{u}(0,X_N)]-\E[\overline{u}(N\Delta t,X_0)]+\overline{u}(N\Delta t,x_0).
\end{align*}
Using Lemma~\ref{lem:Kolmogorov}, one has
\begin{equation}\label{ineq1}
\big|\overline{u}(N\Delta t,x_0)\big|\le e^{-\gamma N\Delta t}\mathcal{N}_r(\varphi)\mathcal{Q}_r(\|x_0\|),
\end{equation}
for some polynomial function $\mathcal{Q}_r$.

Using successively a telescoping sum argument, It\^o's formula and the fact that $\overline{u}$ solves the Kolmogorov equation, one obtains
\begin{align*}
\E[\overline{u}(0,X_N)]-\E[\overline{u}(N\Delta t,X_0)]&=\sum_{n=0}^{N-1}\bigl(\E[\overline{u}(t_N-t_{n+1},X_{n+1})]-\E[\overline{u}(t_N-t_n,X_n)]\bigr)\\
&=\sum_{n=0}^{N-1}\bigl(\E[\overline{u}(t_N-t_{n+1},\tilde{X}(t_{n+1}))]-\E[\overline{u}(t_N-t_n,\tilde{X}(t_n))]\bigr)\\
&=\sum_{n=0}^{N-1}\int_{t_n}^{t_{n+1}}\E[D\overline{u}(t_N-t,\tilde{X}(t)).\bigl(\frac{f(X_n)}{1+\Delta tM_n}-f(\tilde{X}(t))\bigr)]dt,
\end{align*}
where $t_n=n\Delta t$, and the auxiliary process $\bigl(\tilde{X}(t)\bigr)$ is defined by~\eqref{eq:tildeX} and satisfies $\tilde{X}(t_n)=X_n$, with $M_n=\alpha\|f(X_n)\|$.

Introduce the following decomposition:
\[
\E[\overline{u}(0,X_N)]-\E[\overline{u}(N\Delta t,X_0)]=\epsilon_N^1+\epsilon_N^2+\epsilon_N^3,
\]
with
\begin{align*}
\epsilon_N^1&=-\Delta t\int_{0}^{t_N}\E[M_{\ell(t)}D\overline{u}(t_N-t,\tilde{X}(t)).f(X_{\ell(t)})]dt\\
\epsilon_N^2&=\sum_{n=0}^{N-1}\int_{t_n}^{t_{n+1}}\E[D\overline{u}(t_N-t,X_n).\bigl(f(X_n)-f(\tilde{X}(t))\bigr)]dt\\
\epsilon_N^3&=\sum_{n=0}^{N-1}\int_{t_n}^{t_{n+1}}\E[\bigl(D\overline{u}(t_N-t,\tilde{X}(t))-D\overline{u}(t_N-t,X_n)\bigr).\bigl(f(X_n)-f(\tilde{X}(t))\bigr)]dt.
\end{align*}

First, using Lemma~\ref{lem:Kolmogorov} and Theorem~\ref{theo:moment}, one obtains
\begin{align*}
|\epsilon_N^1|&\le \alpha\Delta t\int_0^{t_N}\E[(1+\|\tilde{X}(t)\|^R)\|f(X_n)\|^2]dt\\
&\le Ct_N\Delta t(1+\underset{0\le t\le t_N}\sup~\bigl(\E[\|\tilde{X}(t)\|^{2R}]\bigr)^{\frac12}\bigl(1+\underset{0\le n\Delta t\le T}\sup~\E[\|X_n\|^{4q}])^{\frac12}\\
&\le C\Delta t\mathcal{P}(t_N,\|x_0\|),
\end{align*}
where $\mathcal{P}$ is a polynomial function.

Note that the following result has been used above: Theorem~\ref{theo:moment} gives moment bounds for $X_n$, with $0\le n\Delta t\le t_N$, and it is straightforward to deduce moment bounds of the same type for $\tilde{X}(t)$, with $0\le t\le T$, {\it i.e.}
\[
\underset{0\le t\le T}\sup~\E[\|\tilde{X}(t)\|^m]\le \mathcal{P}_m(T,\|x_0\|).
\]

Second, applying It\^o's formula and a conditional expectation argument gives for all $t\in[t_n,t_{n+1}]$
\begin{align*}
\E[D\overline{u}(t_N-t,X_n).&\bigl(f(X_n)-f(\tilde{X}(t))\bigr)]dt=\E[D\overline{u}(t_N-t,X_n).\bigl(\int_{t_n}^{t}Df(\tilde{X}(s)).\frac{f(X_n)}{1+\Delta tM_n}ds\bigr)]\\
&~+\E[D\overline{u}(t_N-t,X_n).\bigl(\frac12\int_{t_n}^{t}\sum_{k=1}^K D^2f(\tilde{X}(s)).(\sigma_k,\sigma_k)ds\bigr)].
\end{align*}
This yields
\begin{align*}
|\epsilon_N^2|&\le C\sum_{n=0}^{N-1}\int_{t_n}^{t_{n+1}}\int_{t_n}^{t_{n+1}}\E\bigl[(1+\|X_n\|)^R (1+\|\tilde{X}(t))\|^q (1+\|X_n\|^q) \bigr]dsdt\\
&~+C\Delta t\sum_{k=1}^{K}\|\sigma_k\|^2\sum_{n=0}^{N-1}\int_{t_n}^{t_{n+1}}\int_{t_n}^t\E[(1+\|X_n\|^R)(1+\|\tilde{X}(s)\|^q)]dsdt.
\end{align*}
Using the Cauchy-Schwarz inequality and moment bounds for $X_n$ and $\tilde{X}(s)$, one then obtains
\[
|\epsilon_N^2|\le C\Delta t\mathcal{P}(t_N,\|x_0\|),
\]
where $\mathcal{P}$ is a polynomial function.

Finally, using a Taylor expansion, one has
\begin{align*}
|\epsilon_N^3|&\le \sum_{n=0}^{N-1}\int_{t_n}^{t_{n+1}}\E[(1+\|X_n\|^R+\|\tilde{X}(t)\|^R)\|f(\tilde{X}(t))-f(X_n)\|^2]dt\\
&\le C\sum_{n=0}^{N-1}\int_{t_n}^{t_{n+1}}\E[(1+\|X_n\|^{R+2q}+\|\tilde{X}(t)\|^{R+2q})\|\tilde{X}(t)-X_n\|^2]dt\\
&\le C\Delta t\mathcal{P}(t_N,\|x_0\|),
\end{align*}
using moment bounds, Cauchy-Schwarz inequality and the inequality
\[
\E[\|\tilde{X}(t)-X_n\|^4]\le C\Delta t^4\E[\|f(X_n)\|^4]+\E[\|\sigma (B(t)-B(t_n))\|^4]\le C\Delta t^2\mathcal{P}(t_N,\|x_0\|).
\]

Gathering the estimates gives
\begin{equation}\label{ineq2}
\big|\E[\overline{u}(0,X_N)]-\E[\overline{u}(N\Delta t,X_0)]\le \Delta t\mathcal{P}(N\Delta t,\|x_0\|).
\end{equation}

Combining the inequalities~\eqref{ineq1} and~\eqref{ineq2} then concludes the proof of Theorem~\ref{theo:error}.
\end{proof}

\section{Acknoledgments}

This work is partially supported by the SIMALIN project ANR-19-CE40-0016 of the French National Research Agency.


\end{document}